\documentclass[12pt]{amsart}

\usepackage{cite} 

\usepackage{amsmath,amssymb,url}
\usepackage{amsthm}
\usepackage{tikz,color,pgf}
\usepackage{enumerate}

\newtheorem{theorem}{Theorem}[section]
\newtheorem{corollary}[theorem]{Corollary}

\newtheorem{proposition}[theorem]{Proposition}
\newtheorem{conjecture}[theorem]{Conjecture}

\newtheorem{question}[theorem]{Question}
\newtheorem{definition}[theorem]{Definition}
\newtheorem{remark}[theorem]{Remark}
\newtheorem{example}[theorem]{Example}
\newcommand{\ignore}[1]{}

\title{Fibonacci-run graphs II: Degree sequences}

\date{\today}
\author{\"Omer E\u{g}ecio\u{g}lu}
\address{Department of Computer Science, University of California Santa Barbara 
\\Santa Barbara, California 93106, USA }
\email{omer@cs.ucsb.edu}

\author{Vesna Ir\v{s}i\v{c}}
\address{Faculty of Mathematics and Physics, University of Ljubljana
\\Slovenia \newline
Institute of Mathematics, Physics and Mechanics, Ljubljana, Slovenia}
\email{vesna.irsic@fmf.uni-lj.si}

\newcommand{\R}{\mathcal{R}}

\DeclareMathOperator{\inv}{inv}
\newcommand{\half}{{\textstyle\frac12}}

\begin{document}
	
\maketitle

\begin{abstract}
Fibonacci cubes are induced subgraphs of hypercube graphs obtained by 
restricting the vertex set  to 
those binary strings which do not contain consecutive 1s. This class of graphs 
has been studied extensively and generalized in many different directions.
Induced subgraphs of the hypercube on 
binary strings with restricted runlengths as vertices define Fibonacci-run graphs. These graphs have the same number 
of vertices as Fibonacci cubes, but fewer edges and different graph theoretical properties.

Basic properties of Fibonacci-run graphs are presented in 
a companion paper, while in this paper we consider the nature of the 
degree sequences of Fibonacci-run graphs. The generating 
function we obtain is a refinement of the generating function 
of the degree sequences, and has a number of corollaries, obtained as specializations. We also obtain several properties of Fibonacci-run graphs viewed as a partially ordered set, and discuss its embedding properties.
\end{abstract}


\medskip\noindent
\textbf{AMS Math.\ Subj.\ Class.\ (2020)}: 05C75, 05C30, 05C12, 05C40, 05A15

\section{Introduction}
\label{sec:intro}

The \emph{$n$-dimensional hypercube} $Q_n$ is the graph with all binary strings of length $n$ as vertices, 
where two vertices $v_1 v_2 \ldots v_n$ and $u_1 u_2 \ldots u_n$ are 
adjacent if and only if $v_i \neq u_i$ for exactly 
one index $i \in [n]$. We have 
$|V(Q_n)| = 2^n$, and $|E(Q_n)| = n 2^{n-1}$.
 \emph{Fibonacci cubes} $\Gamma_n$ are a subfamily of $Q_n$, and were introduced by Hsu~\cite{hsu1993}. 
The vertices of 
$\Gamma_n$ are the 
\emph{Fibonacci strings} of length $n$, 
$$\mathcal{F}_n = \{ v_1 v_2 \ldots v_n \in \{0,1\}^n ~|~ \; v_i \cdot v_{i+1} = 0 ~, 
i \in [n-1]\}~,$$
and two vertices are adjacent if and only if  they differ in exactly one coordinate. 
Shortly, $\Gamma_n$ is the subgraph of $ Q_n$, 
induced by the vertices that do not contain consecutive $1$s. 
This family of graphs turned out to be interesting, and has been widely investigated, see for example~\cite{klavzar2013-survey, savitha+2020, saygi+2019, azarija+2018, mollard2017}. 

Recall that a Fibonacci cube can equivalently be seen 
by adding the string $00$ to the end of every vertex. 
We call such binary strings {\em extended Fibonacci strings}. 
With this interpretation,
$$
V(\Gamma_n) = \{ w 00 ~|~ w \in \mathcal{F}_n \} ~,
$$ 
and two vertices are adjacent 
if and only if they differ in exactly one coordinate.

Instead of considering extended Fibonacci strings as the vertex set, it is possible to consider \emph{run-constrained binary strings}, which 
are used to define 
\emph{Fibonacci-run graphs} introduced in~\cite{paper1}.
Run-constrained binary strings
are strings of $0$s and $1$s, in which every run  
of $1$s appearing in the word is immediately followed by a strictly longer run of $0$s. 
Run-constrained strings, together with the null word $\lambda$ and the singleton $0$, 
are generated freely (as a monoid) by the letters from the infinite alphabet
\begin{equation}\label{R}
R = 0,100, 11000,1110000, \ldots 
\end{equation}
This means that every run-constrained binary string can be written 
uniquely as a concatenation of zero or more strings from $R$.
Note that run-constrained strings of length $ n \geq 2$ must end with $00$. 

For $n \geq 0$, the Fibonacci-run graph $\R_n$,  has the vertex set
$$
V(\R_n) = 
\{ w 00 ~|~ w00 \mbox{ is a run-constrained string of length $n+2$} \}~,
$$ 
and edge set
$$
E(\R_n) = \{ \{ u00, v00 \} ~|~ H ( u,v ) = 1 \} ~,
$$
where $H(u,v)$ is the Hamming distance between $u, v \in \{0,1\}^n$, i.e.~the number of coordinates in which $u$ and $v$ differ.

\begin{figure}[ht]
	\begin{center}
		\includegraphics[width=0.7\textwidth]{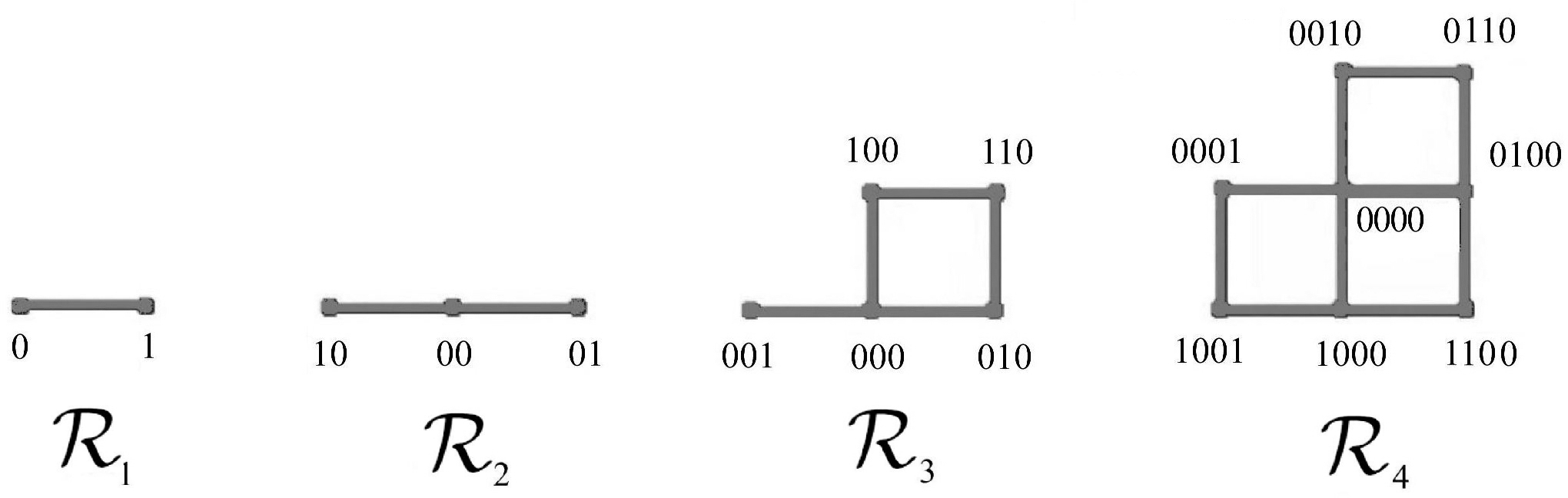}
	\end{center}
	\caption{Fibonacci-run graphs $\R_n$ for $n \in [4]$.}
	\label{fig:fibonacci_run2}
\end{figure}

We take $\R_0$ to be the graph with a single vertex corresponding to the label $00$, 
which after the removal of the trailing pair of zeros,
corresponds to the null word. Clearly, $\R_n$ is a subgraph of $Q_{n+2}$. However it 
is more natural to see it as a subgraph of $Q_n$ after suppressing the trailing $00$ in the 
vertex labels of $\R_n$. In this way, we can view the vertices of $\R_n$ 
without the trailing pair of zeros as 
$$ 
V(\R_n) = 
\{ w ~|~ w00 \mbox{ is a run-constrained binary string of length $n+2$}\} ~.
$$
This is the same kind of convention as viewing $\Gamma_n$ as a subgraph of $ Q_{n+2}$ if 
one thinks of the vertices as extended Fibonacci strings, or as a subgraph of $Q_n$ as usual by 
suppressing the trailing 00 of the vertex labels.
The graphs $\R_1$ -- $\R_4$ with this truncated labeling of run-constrained strings are
shown in Figure~\ref{fig:fibonacci_run2}.

Basic properties of Fibonacci-run graphs
such as the number of vertices, the number of edges, diameter, 
the decomposition into lower dimensional 
Fibonacci-run graphs, Hamiltonicity and the nature of the
asymptotic average degree are studied in~\cite{paper1}.

The rest of the paper is organized as follows.
After the general preliminaries in Section~\ref{sec:prelim},
we consider a decomposition of run-constrained binary strings and prove a 
result on special collections of words in Section~\ref{sec:decomposition}.
In Section~\ref{sec:up-down}, we consider the problem of 
keeping track of both the up-degree and the down-degree of a run-constrained string.
Calculation of the generating function of the up-down degree enumerator polynomials is presented in
Section~\ref{sec:GF}. The proof is divided into a number of subsections.
In Section~\ref{sec:con}, we derive a number of 
consequences of the generating function obtained in Section~\ref{sec:GF}.
Among these is the generating function for the degree enumerator polynomials of 
Fibonacci-run graphs.
Following this, in Section~\ref{sec:poset},
we consider a number of parameters for Fibonacci-run graphs 
as partially ordered sets. These include the rank generating 
polynomial, enumeration of the maximal elements, and the calculation of the 
M\"{o}bius function. A combinatorial aspect of run-constrained strings, namely the 
generating function by inversions is
presented in Section~\ref{sec:inv}.
Embedding and related results are in Section 
\ref{sec:encoding}, followed by conjectures, questions and further directions in Section~\ref{sec:further}.

\section{Preliminaries}
\label{sec:prelim}

In this section, we present definitions and some 
known results which are needed in the paper. 
To avoid possible confusion that may arise due to the initial values, we reiterate that 
\emph{Fibonacci numbers} are defined as 
$f_0 = 0$, $f_1 = 1$, and $f_n = f_{n-1} + f_{n-2}$ for $n \geq 2$.
The \emph{Hamming weight} of a binary string $u$ is 
the number of $1$s in $u$, denoted by $|u|_1$. If $a, b$ are strings, 
then $ab$ denotes the concatenation of these two strings in that order. 
Similarly, for a set of strings $B$, we set $aB = \{ab ~|~ \; b \in B\}$.

The number of vertices and the number of edges of 
$\R_n$ for $ n \geq 5$ are given by 
\begin{eqnarray*}
|V(\R_n)| & = & |V(\Gamma_n)| = f_{n+2} ~, \\
|E(\R_n)| & = & |E(\Gamma_n)| - |E(\Gamma_{n-4})| = 
(3 n + 4 ) f_{n-6} + (5 n + 6 ) f_ {n-5}~,
\end{eqnarray*}
as proved in~\cite[Lemma 3.1]{paper1} and~\cite[Corollary 4.3]{paper1}.


Fibonacci-run graphs can also be 
viewed as partially ordered sets whose structure is inherited from the 
Boolean algebra of subsets of $[n]$. The elements here correspond to 
all binary strings of length $n$ and the covering relation is flipping a 0 to a 1. 
Therefore in $\R_n$ we have a natural distinction between up- and down-degree of a vertex, 
denoted by $ \deg_{up}(v)$ and $\deg_{down}(v)$. Here 
$ \deg_{up}(v)$ is the number of vertices $u$ in $\R_n$ obtained by changing a 0 to a 1, and 
$\deg_{down}(v)$ is the number of vertices $u$ in $ \R_n$ obtained from $v$ by changing a 1 to a 0.
Clearly
$$
\deg(v) = \deg_{up} (v) + \deg_{down}(v) ~.
$$
Note that in $ \R_n$, $\deg_{down}(v)$ is not necessarily 
equal to the Hamming weight of $v$ because of 
the constraints on the run-lengths that must hold. \\

The degree sequences, i.e. the nature of the vertices of a given degree in a graph, 
has been well studied for Fibonacci cubes~\cite{klavzar+2011}. 
Here
we keep track of the degree sequences of our graphs $\R_n$ as the 
coefficients of a polynomial. This polynomial is called 
the \emph{degree enumerator polynomial} of the graph denoted by $g_n(x)$. The 
coefficient of $x^i$ in 
the degree enumerator polynomial is the number of vertices of degree $i$ in $\R_n$. 
More precisely,

\begin{definition}\label{gdefinition}
The degree enumerator polynomials $g_n(x)$ of $\R_n$ is defined for $ n \geq 1$ by 
\begin{equation}\label{gn}
g_n(x) = \sum_{v\in \R_n } x^{\deg(v)}  ~.
\end{equation}
\end{definition}

Similar polynomials are defined to keep track of the up- and down-degree sequences as well.
In particular the \emph{down-degree enumerator polynomial} 
and the \emph{up-degree enumerator polynomial} of $\R_n$
are defined as 
$$
\sum_{v\in \R_n } d^{\deg_{down} (v)}~,  ~~\mbox{ and   } ~~~
\sum_{v\in \R_n } u^{\deg_{up} (v)} ~,
$$
respectively. 

The generating function of the sequence of down-degree enumerator polynomials of $\R_n$ is 
$$
\sum_{n \geq 1} 
t^n \sum_{v\in \R_n } d^{\deg_{down} (v)}  ~.
$$
The generating functions of the sequence of up-degree enumerator polynomials and the 
degree enumerator polynomials are defined similarly.

\begin{figure}[ht]
	\begin{center}
		\includegraphics[width=0.45\textwidth]{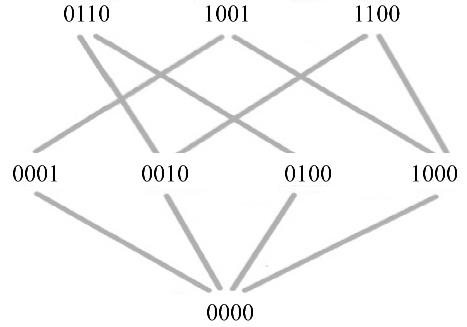}
	\end{center}
	\caption{The Hasse diagram of the Fibonacci-run graph $\R_4$ when viewed as 
a partially ordered set.}
	\label{fig:poset}
\end{figure}

The nature of the distribution of the up-degrees and the down-degrees are most easily seen 
from the Hasse diagram of $\R_n$ for which $ \deg_{up}(v)$ and 
$\deg_{down}(v)$ are simply the number of edges emanating up and down from $v \in \R_n$, respectively.
Inspecting Figure~\ref{fig:poset}, we see that the down-degree, 
up-degree and the degree enumerator polynomials of $\R_4$ are given respectively by
\begin{equation}\label{R4polynomials}
1+4d + 3 d^2 , ~~~~
3+2u+2 u^2 + u^4  , ~~~~
5x^2 + 2 x^3 + x^4 ~.
\end{equation}

It was determined in~\cite[Proposition 7.1]{paper1} and~\cite[Proposition 7.2]{paper1} that 
generating function for down-degree enumerator polynomials of $\R_n$ is
\begin{equation}\label{GFdown}
\frac{t (1 +d +dt +  (d^2-1)t^2  +d(d-1)t^3 + d(d-1)t^4}{1 -t -t^2 -(d-1)t^3 - d(d-1) t^5 }~,
\end{equation}
and the generating function for up-degree enumerator polynomials of $\R_n$ is
\begin{equation}\label{GFup}
\frac{t (1 +u - (u-2) t -2u t^2 + t^3 - (u-1)t^5 - (u-1)t^6 )}
{1- u t - 2 t^2 +(2 u-1) t^3 +t^4 - (u-1) t^5 + (u-1) t^7} ~.
\end{equation}

Our general aim in this paper is to study the generating function
of the bivariate polynomials
\begin{equation}\label{bivariate}
\sum_{v\in \R_n } u^{\deg_{up}(v)} d^{\deg_{down}(v)}
\end{equation}
that we refer to as the {\em up-down degree enumerator}  of $\R_n$.
For example, for $n=4$, this polynomial is given by
\begin{equation}\label{up-down_R4}
3 d^2 + 2du +2 du^2+u^4 ~,
\end{equation}
as can be verified by inspecting $\R_4$ in Figure ~\ref{fig:poset4}, where the term 
contributed by each vertex is indicated in a box. 

\begin{figure}[ht]
	\begin{center}
		\includegraphics[width=0.5\textwidth]{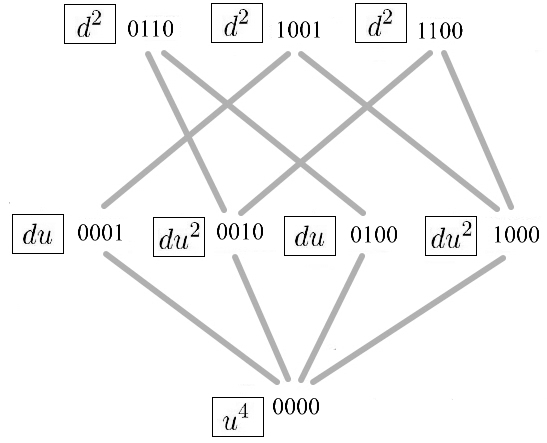}
	\end{center}
	\caption{The calculation of the bivariate up-down degree enumerator polynomial 
$ 3 d^2 + 2 d u + 2 d u^2 + u^4$ of $\R_4$.}
\label{fig:poset4}
\end{figure}

The  up-down degree enumerator polynomial is a refinement of the both up- and down-degree enumerator 
polynomials and of the degree enumerator polynomial $g_n(x)$. 
For instance the polynomial in~\eqref{up-down_R4} specializes to 
the first polynomial in~\eqref{R4polynomials} for $u=1$, to the second one for $ d=1$, and to 
the last for $ u=d=x$.
More generally, the generating function of the up-down degree enumerator polynomials for 
$\R_n$ specializes to the generating functions~\eqref{GFdown} and~\eqref{GFup} as corollaries, 
and provide the generating function for the degree enumerator polynomials $\{ g_n(x) \}_{n \geq 1}$ 
itself for $ u=d=x$. 

We define the generating function $GF$ 
of the up-down degree enumerator polynomials of $\R_n$ 
formally as follows:

\begin{definition}\label{GFdefinition}
The generating function of the sequence of the up-down degree enumerator polynomials 
is defined as
\begin{equation}\label{GF}
GF = GF(u,d; t) = \sum_{n\geq 1} t^n \sum_{v\in \R_n } u^{\deg_{up}(v)} d^{\deg_{down}(v)} ~.
\end{equation}
\end{definition}

\begin{example}
For the graphs $ \R_1 $ through $\R_8$, the 
up-down degree enumerator polynomials are as follows:

	\begin{eqnarray*}
		&& d + u \\
		&& 2 d + u^2 \\
		&& d + d^2 + 2 d u + u^3 \\
		&& 3 d^2 + 2 d u + 2 d u^2 + u^4 \\
		&& 5 d^2 + 2 d^2 u + 3 d u^2 + 2 d u^3 + u^5 \\
		&& 4 d^2 + 2 d^3 + 6 d^2 u + 2 d^2 u^2 + 4 d u^3 + 2 d u^4 + u^6 \\
		&& 3 d^2 + 7 d^3 + 5 d^2 u + 9 d^2 u^2 + 2 d^2 u^3 + 5 d u^4 + 2 d u^5 + u^7 \\
		&& 2 d^2 + 10 d^3 + d^4 + 4 d^2 u + 8 d^3 u + 7 d^2 u^2 + 12 d^2 u^3 + 2 d^2 u^4 + 6 d u^5 + 2 d u^6 + u^8 
	\end{eqnarray*}
\end{example}

\section{Decomposition of run-constrained strings and a preparatory result}
\label{sec:decomposition}
In this section, we present some preliminary results about formal power series, 
which are needed for the computation of our  up-down degree enumerators and their
generating function $GF$.

We also need another description of run-constrained 
binary strings. Let 
$$
S = \{ 100, 11000, 1110000, \cdots\} ~.
$$
Every run-constrained binary string consists of words from $S$ 
interspersed with runs of $0$s, including a prefix and a 
suffix which may also be runs of $0$s. 
Let $s^*$ denote an arbitrary string of zero or more words from $S$, 
and $s^+$ denote an arbitrary string of one or more words from $S$. 
So we have $s^+ = s s^*$, and as an example, $s^2 s^*$ denotes all 
strings consisting of two or more words from $S$. 
Every non-trivial run-constrained binary string 
can be written in the form
$$
0^{i_0} s^+ 0^{i_1} s^+ 0^{i_2} \cdots 0^{i_k} s^+ 0^{i_{k+1}},
$$
where $ k \geq 0$, $i_0, i_{k+1} \geq 0$, 
$i_1 , i_2, \ldots , i_k  \geq 1 $. The runs 
$0^{i_1},   0^{i_2}, \ldots , 0^{i_k}$ are called 
{\em internal runs}, the initial string $0^{i_0}$ is the {\em pre-run}, 
and the final string $0^{i_{k+1}}$ is the {\em post-run} of the word. 
Note that for the latter two, 
we do not rule out the possibility that they have length 
zero (i.e.~$i_0 = 0 $ or $i_{k+1} = 0 $.)
So in that sense they are not ``real'' runs like 
the interior runs of the string, which are the portions in between the letters of $S$ that appear in the 
word, and must have positive length. 

Note that the up-down generating function $G$ of the words in $S$ itself is
\begin{equation}\label{G}
G= d t^3 + d^2 t^5 + d^2 t^7 + \cdots = dt^3+ \frac{d^2 t^5}{1-t^2} ~.
\end{equation}

We also make us of the following two preparatory results. Consider the alphabet 
$\Sigma = \{a,b\}$ and let $ \Sigma^*$ denote all words over $\Sigma$. 
The length of 
$u \in \Sigma^*$ is denoted by $|u|$. Let $|u|_a$ and $|u|_b$ denote the number of 
occurrences of $a$ and $b$ in $ u $, respectively. Let also
$|u|_{aa}$,
$|u|_{ab}$,
$|u|_{ba}$,
$|u|_{bb}$ denote the number of appearances of the words 
$aa, ab, ba, bb$ in $u$, respectively.
For $ n \geq 0$, let 
$ a \Sigma^n a = \{ a w a ~|~ w \in \Sigma^*, ~ |w|=n \}$. Similarly, we define 
the sets of strings
$ a \Sigma^n b$,
$ b \Sigma^n a$, and
$ b \Sigma^n b$.

For a word $u$ with $|u| \geq 2 $, define
\begin{equation}\label{monomial}
m(u) = x^{|u|_a} y^{|u|_b}
\alpha^{|u|_{aa} +|u|_{ba}}
\beta^{|u|_{ab} +|u|_{bb}} ~.
\end{equation}

\begin{example}
For the word $u = aababbaaa \in a \Sigma^7 a $, we have 
$|u| = 9$, 
$|u|_a = 6$, 
$|u|_b = 3$, 
$|u|_{aa} + |u|_{ba} = 5$, 
$|u|_{ab} + |u|_{bb} = 3$, and consequently 
$$
m(u) = x^6 y^3  \alpha^5 \beta^3 ~.
$$
\end{example}

Next, we prove the following proposition, to be used for the calculation of the  
generating function $GF$ of the  up-down degree enumerator polynomials.

\begin{proposition}
	\label{prop:words}
	Let $n \geq 0$ be an integer and let $w \in \Sigma^*$, where $\Sigma = \{a,b\}$. Then
	\begin{eqnarray}\label{sum1}
	\sum_{|w|=n } m(awa) &=& \alpha x^2 (\alpha x + \beta y )^n , \label{sum2} \\
	\sum_{|w|=n}  m(awb) &=& \beta x y  (\alpha x + \beta y )^n ,\label{sum3} \\
	\sum_{|w|=n}  m(bwa) &=& \alpha x y  (\alpha x + \beta y )^n ,\label{sum4}\\
	\sum_{|w|=n } m(bwb) &=& \beta y^2 (\alpha x + \beta y )^n .
	\end{eqnarray}
\end{proposition}

\begin{proof}
Consider the first identity. For $ n =0$, 
there is only one word $ aa$, and both sides are 
$\alpha x^2$ in this case. 
If $ n >0$ and $ u = a w a$, then  note that the number 
$|u|_{aa} +|u|_{ba} $ is the number of $a$'s in $ w a$ and 
$|u|_{ab} +|u|_{bb}$ is the number of $b$'s in $w$. Given a word $w$ with 
$|w|_a = k $ and $|w|_b = n-k$, we calculate $m(u)$ as 
$$
m(u) =  x^{k+2} y^{n-k} \alpha^{k+1} \beta^{ n-k} =
\alpha x^2 x^{k} y^{n-k} \alpha^{k} \beta^{ n-k} ~.
$$
Since there are ${n \choose k }$ such strings $w$, we obtain
$$
\sum_{|w|=n } m(a ua ) = \alpha x^2 
\sum_{k=0}^n {n \choose k }
x^{k} y^{n-k} \alpha^{k} \beta^{ n-k} =
\alpha x^2 (\alpha x + \beta y )^n  ~.
$$
The proofs of the other identities are similar.
\end{proof}

By summing each identity in 
Proposition~\ref{prop:words} over all 
nonnegative integers $n$, we obtain the following formulas.

\begin{corollary}
	\label{cor:words}
	Let $\Sigma = \{a,b\}$ and $m$ be as defined in~\eqref{monomial}. Then
	\begin{eqnarray}\label{GFsum1}
	\sum_{w \in \Sigma^*} m(awa) &=& \frac{\alpha x^2}{1- (\alpha x + \beta y )} , \label{GFsum2}\\
	\sum_{w \in \Sigma^*}  m(awb) &=& \frac{\beta x y }{1- (\alpha x + \beta y )} , \label{GFsum3}\\
	\sum_{w \in \Sigma^*}  m(bwa) &=& \frac{\alpha x y}{1- (\alpha x + \beta y )} , \label{GFsum4}\\
	\sum_{w \in \Sigma^*} m(bwb) &=& \frac{\beta y^2}{1- (\alpha x + \beta y )} .
	\end{eqnarray}
\end{corollary}

\section{Up-down degree polynomials}
\label{sec:up-down}

In the general case we aim to keep track of the terms
\begin{equation}\label{stats}
u^{\deg_{up}(v) } d^{\deg_{down}(v)} t^{|v|} ~
\end{equation}
for every run-constrained binary string $v$. If we are to only keep 
track of the down-degree of a run-constrained binary string, then the problem is considerably 
simpler. In this case we are flipping $1$s in the string 
to $0$s, and the contribution of every word $s \in S$ in the 
observed run-constrained binary string $v$, independently of where 
it is located in $v$, is either 1 ($s = 100$) or 2 ($s \in S \setminus \{100\}$). 

The difficulty with keeping track of the up-degree arises in the following situations.
As an example
consider the subword $ s_1 0 s_1$ that appears somewhere 
in the string, where $s_1 = 100$, and use parentheses to highlight the words from $S$:
$$
\cdots (100) \, 0 \, (100) \cdots
$$
The 0 in the middle can be flipped to 1 in the case that the 100 on the right is 
followed by a  0:
$$
\cdots (100) \, 0 \, (100)  \, 0\cdots
$$
but not if the 100 on the right is followed by another word from $S$:
$$
\cdots (100) \, 0 \, (100)  \, (11 000) \cdots
$$
As another example, consider the last subword $s \in S$ in the word. It may be followed by zero or more $0$s. 
For example,
$$
\cdots (100), ~~
\cdots (100) \, 0, ~~
\cdots (100) \, 00, ~~
\cdots (100) \, 000,   ~~
\cdots (100) \, 0000,  \ldots
$$
In the first two cases, the 0 immediately to the right of 1 cannot be flipped to a 1 but in all 
other cases it can be. 
Additionally, whenever we have $r \geq 3$ trailing $0$s in a case like this, $r-2$ of 
those can be flipped to a $1$.

\section{Calculation of $GF$}
\label{sec:GF}
\begin{theorem}
	\label{thm:up-down}
	The generating function for the up-down degree enumerator polynomials of the graphs $\R_n$ is 
given by
	$GF = N_{u,d}/D_{u,d}$ where
{\small 
	\begin{eqnarray*}
		N_{u,d} &=& 
		(d+u)t -d (u-2)t^2 +(d^2-d-2u)t^3 -(d-2)d (u-2) t^4 \\
&& -(d-1)(u-d+d u)t^5 -d (d+u-2) t^6 +d (1-2 d + 2 d^2 +d u-2 d^2 u) t^7 \\
&& -2(d-1)d^2 (u-1) t^8 -(d-1)d^2(d+1) (u-1) t^9  -(d-1)^2 d^2 (u-1) t^{10} \\
&& -(d-1)^2 d^2 (u-1) t^{11} ~,
	\end{eqnarray*}
}
	and 
{\small
	\begin{eqnarray*}
		D_{u,d} &=& 
		1 -u t -2 t^2 +(2u -d)t^3 + t^4 +(2d -d^2 -u) t^5 +d(d u -1) t^7 \\
		&& + 2 (d-1) d^2 (u-1) t^9 +(d-1)^2 d^2 (u-1) t^{11} ~.
	\end{eqnarray*}
}
\end{theorem}
\begin{proof}	
	We consider the cases according to the type of word from $S$ 
that precedes the leftmost internal run of $0$s,  and the 
word from $S$ that follows the rightmost internal run of $0$s. There 
is also the case, where the run-constrained binary string does not have 
internal runs of $0$s. Thus	there are altogether five cases to consider. In 
the cases $A)$ through $D)$,
	$ k \geq 1$, ${i_0}, i_{k+1} \geq 0$, and $ i_1, i_2, \ldots, i_k >0$. Note that $s$ simply denotes a word from $S$, and this word may be different at different places in the below schematic description of the 
cases. Similarly, $s^2 s^*$ denotes a string of at 
least two words from $S$.
	\begin{enumerate}
		\item[$A)$]
		The string is of the form $ 0^{i_0} \, s \, 0^{i_1} \cdots 0^{i_k} \, s \, 0^{i_{k+1}}$ 
		\item[$B)$]
		The string is of the form $ 0^{i_0} \, s^2s^* \, 0^{i_1} \cdots 0^{i_k} \, s \, 0^{i_{k+1}}$ 
		\item[$C)$]
		The string is of the form $ 0^{i_0}\,  s\,  0^{i_1} \cdots 0^{i_k} \, s^2s^* \, 0^{i_{k+1}}$ 
		\item[$D)$]
		The string is of the form $ 0^{i_0} \, s^2 s^* \, 0^{i_1} \cdots 0^{i_k} \, s^2s^*\,  0^{i_{k+1}}$ 
		\item[$E)$] The string has no internal runs of $0$s.
	\end{enumerate}
	Let $GF_A$, $GF_B$, $GF_C$, $GF_D$, $GF_E$ denote the generating functions of the 
classes of strings in the cases~$A)$ through~$E)$, respectively.
	Then the generating function for the run-constrained binary strings 
	with the statistic~\eqref{stats} is 
	$$
	GFX = 
	GF_A+ GF_B+ GF_C+ GF_D+ GF_E,
	$$
	and the generating function $GF$ for the same statistic for $\R_n$ is (after getting rid of the strings $0$ and $00$, 
and removing the trailing $00$ in all strings)
\begin{equation}\label{GFX}
GF = 	(GFX- t - t^2)/t^2 ~.
\end{equation}
	
	To simplify the notation in the remaining part of the proof, we define 
the following quantities:
		\begin{eqnarray*}
			G& =&  dt^3+ \frac{d^2 t^5}{1-t^2}, \\
			\alpha & = & \frac{u t}{1-u t} , ~~~\beta  =  \frac{ t}{1-u t}, \\
			x & = & G , ~~~y  =  \frac{G^2}{1-G} ~.
		\end{eqnarray*}

\noindent
{\bf Calculation of $GF_E$}\\
	We first consider the calculation of $GF_E$, which is the most straightforward. 
	Since there are no internal runs of $0$s in case~$E)$, 
	we can partition the strings in~$E)$ into three 
classes and calculate the generating function for each.
	\begin{enumerate}
		\item
		{\em The extended Fibonacci string is all $0$s:}
		
		The generating function for these is 
		\begin{equation}\label{E1}
		t + \frac{t^2}{1- u t} ~.
		\end{equation}
		
		\item {\em The extended Fibonacci string has a single word from $S$}: 
		
		The generating function is 
		\begin{equation}\label{E2}
		\left(
		1 + \frac{2t}{1-u t} + \frac{t^2}{(1- u t)^2}
		\right) G ~.
		\end{equation}
		To see this,
		note that 
		words  in this set  
		are of the form
		$ 0^i s 0^j$ with $i,j \geq 0$.
		For $j=0$ the generating function is the product of $G$ and 
		$$
		1 + t + t^2 + u t^3 + u^2 t^4 + \cdots = 1+t+\frac{t^2}{1- u t} ~.
		$$
		For $j=1$ it is the product of $G$ and 
		$$
		t + u t^2 + u t^3 + u^2 t^4 + \cdots = t + u t^2 + \frac{u t^3}{1-u t} ~.
		$$
		For $j=2$ it is the product of $G$ and 
		$$
		u t^2 + u^2 t^3 + u^2 t^4 + u^3 t^5 + \cdots = 
u t^2 +  u^2 t^3 + \frac{u^3t^5}{1- u t} ~,
		$$
		and so on.
		Adding the fractional expressions over $j$ yields
		\begin{equation}\label{E11}
		\sum_{j \geq 0} \frac{ u^j t^{j+2}}{1-u t} = \frac{t^2}{1-ut} 
		\sum_{j\geq 0} u^j t^j =\frac{t^2}{(1- u t)^2 } ~.
		\end{equation}
		The sum of the pairs of terms $1+t$, $t + ut^2$, $ut^2 + u^2 t^3$, etc. 
over $j \geq 0 $ is calculated to be 
		\begin{equation}\label{E12}
		1 + \frac{2t}{1-u t}  ~.
		\end{equation}
		Adding the contributions of~\eqref{E11} and~\eqref{E12} proves~\eqref{E2}.
		
		\item {\em The extended Fibonacci string has two or more words from $S$}: 
		
		We prove that in this case the generating function is 
		\begin{equation}\label{E3}
		\left(
		1 + t + \frac{t+t^2}{1- u t} + \frac{t^2}{1-u t} + \frac{t^3}{(1- u t)^2} 
		\right)  \frac{G^2}{1-G} ~.
		\end{equation}
		The strings here are of the form
		$0^i s^2 s^* 0^j$ with $i,j \geq 0$. The generating function of 
		the strings of at least two words from $S$ is $ G^2/(1-G)$.
		
		We again calculate the contribution to the generating function for 
$j = 0, 1, 2, \ldots$.
		For $j=0$, we have 
		$$
		1+t + t^2 + u t^3 + u^2 t^4 + \cdots = 1+ t + \frac{t^2}{1- ut},
		$$
		for $j=1$
		$$
		t + t^2 + t^3 + ut^4 + u^2 t^5 + \cdots = t + t^2 + \frac{t^3}{1-u t},
		$$
		for $j=2$
		$$
		u t^2 + u t^3 + u t^4 + u^2 t^5 + \cdots = 
u t^2 + u t^3 + \frac{u t^4}{1-u t}~,
		$$
		etc. Adding the fractional terms gives
		$$
		\frac{t^2}{1- ut } + \sum_{j \geq 1 } \frac{ u^{j-1} t^{j+2}}{1- u t} = 
		\frac{t^2}{1-u t} + \frac{t^3}{(1- u t)^2} ~.
		$$
		The sum of the pairs of remaining terms that appear for each $j$ is 
		$$
		1 + \frac{t}{1- ut} + t + \frac{t^2}{1- u t}
		$$
		and adding these up gives~\eqref{E3}.
	\end{enumerate}
	
	Finally, adding up and simplifying the contributions of~\eqref{E1}, \eqref{E2} 
and \eqref{E3}, we obtain
		\begin{eqnarray}\label{GFE}
		GF_E & =  & 
		\frac{t(1+(1-u)t)}{1- u t} + 
		\frac{(1+(1- u)t)^2}{(1- u t)^2} G \\ \nonumber 
&& \hspace{2cm} +~ 
		\frac{(1+(1- u)t )(1+(1-u)t + (1-u) t^2)}{(1- u t)^2} 
		\frac{G^2}{1-G} ~.
		\end{eqnarray}

\noindent
{\bf Calculation of $GF_A$}\\
	We first consider the generating function
	$GF_A$ on an example. We will see that the observations made on it can be 
used on a general string. Take a word of the type
	\begin{equation}\label{typeA}
	0^{i_0} s 0^{i_1} s 0^{i_2} s^2 s^* 0^{i_3} s^2 s^* 0^{i_4} s 0^{i_5} s 0^{i_6}~.
	\end{equation}
	Here
	$k=5$, $i_0, i_6 \geq 0$, $i_1,i_2,i_3,i_4,i_5 > 0$.
	The contribution of the 
	pre-run in such a word is the factor
	$$
	1+ ut + \frac{u t^2}{1- u t} ~.
	$$
	Note that we are using the fact that the pre-run is followed by $s 0$.
	The contribution of the first interior run $0^{i_1}$ is 
	$$
	\alpha = \frac{ut}{1-ut}~
	$$
	as it is located in the context $ s 0^{i_1} s 0$.
	The contribution of the second interior run
	$0^{i_2}$ is 
	$$
	\beta = \frac{t}{1- ut}
	$$
	as it appears in the context $ s 0^{i_2} s^2 s^*$.
	Continuing, 
	the contribution of the third interior run
	$0^{i_3}$ is  
	$$ 
	\beta = 
	\frac{t}{1- ut}
	$$
	because
	it appears in the context $ s^2 s^* 0^{i_3} s^2 s^*$.
	The contribution of the fourth interior run
	$0^{i_4}$ is  
	$$
	\alpha= \frac{ut}{1- ut}
	$$
	as it appears in the context $ s^2 s^* 0^{i_4} s$.
	
	We can summarize the situation with the contribution of the internal runs as follows:
	\begin{eqnarray*}
		&& s 0^i s 0   \mbox{ and } s^2 s^*  0^i s0 \mbox{ contribute }  \alpha,\\
		&& s 0^i s^2s^*   \mbox{ and } s^2 s^*  0^i s^2s^* \mbox{ contribute }  \beta. 
	\end{eqnarray*}
	Of course each occurrence of $s$ contributes $x=G$ and each 
occurrence of $ s^2 s^*$ contributes 
	$y= G^2 / (1-G)$.
	For our example, this leaves the contribution of the last interior run and the post-run.
	Note that the last interior run contributes $\beta$ if the post-run has length zero, 
	and $\alpha$ is if the post-run has positive length. 
	In the first case the contribution of the post-run is 1, and in the second it is 
	$$
	\frac{t}{1- ut}~.
	$$
	Adding the two contributions, the last interior run and the post-run together contribute
	$$
	\beta + \alpha \frac{t}{1-ut} = \alpha \frac{u^{-1}}{1- u t} ~.
	$$
	Therefore we can ``charge" $\alpha$ as the contribution of the last 
interior run if we make the contribution of the post-run equal to
	$$
	\frac{u^{-1}}{1- u t} ~.
	$$
	We can simplify the situation further. Encode the strings of type~\eqref{typeA}  
	by the word 
	$$
	aabbaa
	$$
	over the two letter alphabet $ \{ a, b \}$, ignoring the runs of $0$s altogether, and 
	encoding $s$ by $a$ and $ s^2 s^*$ by $b$. Since we 
are in case~$A)$, these words start and end with the letter $a$. 
Then we have the following:
	\begin{enumerate}
		\item The contribution of the pre-run is 
		$$
		1+ ut + \frac{u t^2}{1- u t} .
		$$
		\item The contribution of the post-run is 
		\begin{equation}\label{post-run}
		\frac{u^{-1}}{1- u t} ~.
		\end{equation}
		\item The contribution of each letter $a$ is $x = G$.
		\item The contribution of each letter $b$ is $y = G^2/(1-G)$.
		\item The contribution of each letter pair $aa$ or $ba$ is $\alpha$.
		\item The contribution of each letter pair $bb$ or $ab$ is $\beta$.
	\end{enumerate}
	In our example $aabbaa$, there are four $a$'s, two $b$'s, two $aa$'s, one $ba$, 
	one $bb$ and one $ab$. So the generating function of the words encoded by 
	$aabbaa$ is 
	$$
	\left(1+ ut + \frac{u t^2}{1- u t} \right) \left(\frac{u^{-1}}{1- u t} \right) 
	x^4  y^2 \alpha^3 \beta^2 ~.
	$$
	Note that 
	$
	x^4  y^2 \alpha^3 \beta^2 = m(aabbaa)
	$ 
	as defined in~\eqref{monomial}. By~\eqref{sum1},
	$$
	\sum_{|w|=4 } m(awa) = \alpha x^2 (\alpha x + \beta y )^4 ~.
	$$
	Generalizing this results and using~\eqref{GFsum1}, the generating function $GF_A$ is 
	given by
	\begin{equation}\label{GFA}
	GF_A =
	\left(1+ ut + \frac{u t^2}{1- u t} \right) \left(\frac{u^{-1}}{1- u t} \right)  \frac{\alpha x^2}{1- (\alpha x + \beta y)}.
	\end{equation}

\noindent
{\bf Calculation of $GF_B$}\\
	Let us consider how this case differs from the computation of $GF_A$. This time the 
	contribution of the pre-run is 
	$$
	1+t+t^2 + ut^3 + u^2 t^4 + \cdots = 1+t + \frac{t^2}{1-ut} ~.
	$$
	Again the contribution of the post-run is taken to be~\eqref{post-run}. 
Each adjacent pair 
	$aa$ or $ba$ contributes $\alpha$,
	and each pair $bb$ or $ab$ contributes  $\beta$.
	In this case the encoding words start with $b$ and end with $a$.
	Using~\eqref{sum3} and~\eqref{GFsum3}, we find 
	\begin{eqnarray}\label{GFB}
	GF_B & = & 
	\left(1+t + \frac{t^2}{1-ut} \right)
	\left(\frac{u^{-1}}{1- u t} \right) \sum_{n \geq 0} \alpha x y ( \alpha x  + \beta y)^n
	\\ \nonumber
	& = & 
	\left(1+t + \frac{t^2}{1-ut} \right)
	\left(\frac{u^{-1}}{1- u t} \right)  \frac{\alpha x y }{1- (\alpha x + \beta y )} ~.
	\end{eqnarray}

\noindent
{\bf Calculation of $GF_C$}\\
	Here the contribution of the pre-run is as in case~$A)$, but the 
contribution of the post-run is 
	$$
	1 + \frac{t}{1- u t} ~.
	$$ The encoding words over $\{a,b\}$ start with $a$ and end with $b$. Therefore
	by~\eqref{GFsum2}
	\begin{equation}\label{GFC}
	GF_C = 
	\left(1+ ut + \frac{u t^2}{1- u t} \right)
	\left(1 + \frac{t}{1- u t} \right) 
	\frac{\beta x y }{1- (\alpha x + \beta y )} ~.
	\end{equation}

\noindent
{\bf Calculation of $GF_D$}\\
	In this case the contribution of the pre-run is such as in case~$B)$, and the contribution of the post-run is as in 
case~$C)$. The encoding words over $\{a,b\}$ 
start and end with $b$. Therefore by~\eqref{GFsum4}, we have 
	\begin{equation}\label{GFD}
	\left(1+t + \frac{t^2}{1-ut} \right)
	\left( 1 + \frac{t}{1- u t} \right) 
	\frac{\beta y^2 }{1- (\alpha x + \beta y)}~.
	\end{equation}
	
	Finally, adding up the contributions from~\eqref{GFA}, \eqref{GFB}, \eqref{GFC}, 
\eqref{GFD}, \eqref{GFE}  (by Mathematica) 
	gives the generating function $GFX$ and via~\eqref{GFX}, the generating function 
$GF$ given in the theorem.
\end{proof}

\section{Consequences of the up-down degree enumerator}
\label{sec:con}

The degree enumerator 
polynomials $g_n(x)$ for $\R_1$ through $\R_{10}$ (computed by Mathematica) are as follows:

	\begin{eqnarray} \nonumber
	g_1(x) &=& 2x \\ \nonumber
	g_2(x) &=& x^2+2x \\ \nonumber
	g_3(x) &=& x^3+3x^2+x \\ \nonumber
	g_4(x) &=& x^4+2x^3+5x^2 \\ \nonumber
	g_5(x) &=& x^5+2x^4+5x^3+5x^2\\ \label{degree_polynomials} 
	g_6(x) &=& x^6+2x^5+6x^4+8x^3+4x^2\\  \nonumber
	g_7(x) &=& x^7+2x^6+7x^5+9x^4+12x^3+3x^2\\ \nonumber
	g_8(x) &=& x^8+2x^7+8x^6+12x^5+16x^4+14x^3+2x^2\\ \nonumber
	g_9(x) &=& x^9+2x^8+9x^7+15x^6+22x^5+24x^4+14x^3+2x^2\\ \nonumber
	g_{10}(x) &=& x^{10}+2x^9+10x^8+18x^7+30x^6+32x^5+39x^4+10x^3+2x^2 \\ \nonumber
	\end{eqnarray}

Making the substitutions $ u \rightarrow x$ and $ d \rightarrow x$ in the 
generating function of the 
up-down degree enumerator polynomials, 
we find the generating function of the degree enumerator 
polynomials of the $\R_n$ to be as follows.

\begin{theorem}\label{degreegf}
	The generating function  for the degree enumerator polynomials of Fibonacci-run graphs
	$$
	f(t,x) = \sum_{n \geq 1} 
	g_n(x) t^n 
	$$
	is given in closed form by $ N_x/D_x $, where
	\begin{eqnarray*}
		N_x & = & x t 
		\Big( 2 -(x-2)t +(x-3)t^2 -(x-2)^2 t^3 - x (x-1) t^4 
		-2 (x-1) t^5   \\ 
		&& ~~~~  - (x-1)(2 x^2-x+1) t^6  
		-2 x (x-1)^2 t^7
		-x (x-1)^2 (x+1) t^8 \\
		&&~~~~ ~~~~~ - x (x-1)^3 t^9 - x (x-1)^3 t^{10} \Big) 
	\end{eqnarray*}
	and 
	\begin{eqnarray*}
		D_x & = & 1 - x t -2 t^2 + x t^3 + t^4 -x(x-1)t^5 
		+x (x-1)(x+1) t^7 \\
		&& ~~~~ + 2 x^2 (x-1)^2 t^9 + x^2 (x-1)^3 t^{11}.
	\end{eqnarray*}
\end{theorem}

\begin{remark}\label{degree_remark}
	To find the generating function of the number of vertices with degree $k$, 
	we can take $\frac{d^k}{dx^k} f(t,x)$, and then set $x =0$. The 
resulting series divided by $ k!$ is then the 
	generating function of the number of vertices of degree $k$ 
in $ \R_n$. We can denote this series by 
	\begin{equation}\label{kthderivativef}
	\frac{1}{k!}  D^k f(t,x) ~|_{x=0} ~.
	\end{equation}
	In other words, the coefficient of $t^n$ in~\eqref{kthderivativef} is the number of vertices 
of degree $k$ in $\R_n$.
\end{remark}

In the following examples we calculate the number of vertices of small 
degree in Fibonacci-run graphs.\\

\begin{example}
	\label{ex:deg2}
	We differentiate the generating function $f(t,x)$ of Theorem~\ref{degreegf} 
	with respect to $x$ twice using Mathematica, and then 
	set $ x=0$. Dividing the resulting expression by 2 gives
	$$
	\frac{1}{2}  D^2 f(t,x) ~|_{x=0} =
	t^2 + 3 t^3 +5 t^4 + 5 t^5 + 4t^6 + 3 t^7 + \frac{2 t^8}{1-t} ~,
	$$
	confirming that for $ n \geq 8$, $ \R_n$ has exactly two vertices of degree 2.
\end{example}

\begin{example}
	\label{ex:deg3}
	Continuing computing with Mathematica, we find
	$$
	\frac{1}{6}  D^3 f(t,x) ~|_{x=0} =
	t^3 + 2 t^4 +5t^5 + 8 t^6 + 12 t^7 + 14 t^8 + 14 t^9 + 10 t^{10} +
	\frac{10 t^{11}}{1-t^2} + \frac{8 t^{12}}{1-t^2}
	$$
	which means that for $n\geq 11$, the number of vertices of degree 3 in $\R_n$ is 10  if $n$ 
	is odd, and 8 if $n$ is even.
\end{example}

\begin{example}
	\label{ex:deg4}
	For $k=4$, we get
	\begin{eqnarray*}
		\frac{1}{24}  D^4 f(t,x) ~|_{x=0} & = &
		t^4 + 2 t^5 + 6 t^6 + 9 t^7 + 16 t^8 + 24 t^9 + 39 t^{10} + 42 t^{11} + 46 t^{12} \\
		&& ~~~~ + 39 t^{13} + 43 t^{14}  + 
		\frac
		{t^{15} ( 39 + 45 t - 35 t^2 -42 t^3 )}
		{(1-t^2)^2 }
	\end{eqnarray*}
	from which we compute that for $ n \geq 15$, the number of vertices of degree 4 in 
	$\R_n$ is $2n+9$ if $n$ is odd and $3n/2 + 21 $ if $n$ is even.
\end{example}

\begin{example}
	\label{ex:deg5}
Similar  calculations give that 
the number of vertices of degree $5$  in $ \R_n$ for $n\geq 18$ is 
$9n-1$ if $n$ is odd, and $ 12n-48$ if $n$ is even.
\end{example}

The presented examples lead to the following conjecture.

\begin{conjecture}\label{conjecture:GF}
We conjecture that for  a given $k$, the generating function of the number of 
vertices of degree $k$ in $ \R_n$ is of the form
$$
\frac{p_k(t)}{ (1 - t^2 )^{k+1} }
$$
where $ p_k(t)$ is a polynomial of degree $\half (15k +8)$ if $k$ is even, and of degree
$\half (15k+7)$ if $k$ is odd. 
\end{conjecture}

We already have the first few degree enumerator polynomials 
as given in~\eqref{degree_polynomials}. From 
the denominator of their generating function in Theorem~\ref{degreegf}, 
we get the following result.
\begin{theorem}
	If $g_n = g_n(x)$ is the degree enumerator polynomial for $ \R_n$ as defined in Definition 
\ref{gdefinition} with the 
	initial values given by~\eqref{degree_polynomials}, then for $ n \geq 12$
\begin{eqnarray*}
g_n & = & x g_{n-1} +2 g_{n-2} -x g_{n-3} - g_{n-4} +x(x-1) g_{n-5} -x (x^2-1) g_{n-7} \\
&& -2 x^2 (x-1)^2 g_{n-9} - x^2 (x-1)^3 g_{n-11} ~.
\end{eqnarray*}
\end{theorem}

Two other specializations of the generating function of the up-down degree enumerator polynomials are obtained in the following way.
\begin{enumerate}
\item Setting $u=1$, we obtain
the generating function of the down-degree enumerator polynomials  in~\eqref{GFdown}, 
which we had already computed~\cite[Proposition 7.1]{paper1}.
\item Setting $d=1$, we obtain the generating function of the up-degree enumerator polynomials shown in~\eqref{GFup}, also already computed in~\cite[Proposition 7.2]{paper1}.
\end{enumerate}

Note that Theorem~\ref{thm:up-down} can be used to recalculate another result.
The down-degree enumerator generating function differentiated with respect to $d$ is 
	$$
	\frac{
		t (1-t^2) ( 1+ (2d-1)t^2)
	}
	{
		(1 - t - t^2 -(d-1)t^3 - d (d-1) t^5)^2
	}
	$$
	which for $ d=1$ gives the generating function of the number of edges of $ \R_n$, already 
determined in~\cite{paper1} as 
	\begin{equation}\label{egf}
	\frac{t(1-t^4)}{(1-t-t^2)^2} ~.
	\end{equation}
Clearly, this result could also be obtained 
by differentiating the up-degree enumerator generating function with respect to $u$, and then evaluating it
at $u=1$.

\section{Fibonacci-run graphs as partially ordered sets}
\label{sec:poset}

Naturally, one might view the vertices of a Fibonacci-run graph as a partially ordered set 
(\emph{poset} for short). This poset is defined as $(\R_n, \leq )$, 
where the covering relation 
is given as follows: for $ u, v \in \R_n$, $v$ covers $u$ if and only if 
$v$ is obtained from $u$ by flipping a $0$ in $u$ to a $1$. 
The binary relation ``$\leq$" is 
the transitive closure of this covering relation. 
We have shown the Hasse diagram of $\R_4$ in Figure~\ref{fig:poset}.
Figure~\ref{fig:poset6} depicts the Hasse diagram of $\R_6$.

Note that the number of vertices $v$ that cover a vertex $ u\in \R_n$ is precisely the 
up-degree
$\deg_{up} (u)$. 
The number of vertices $ u \in \R_n$ that are covered by  $ v$ is the 
down-degree $\deg_{down} (v)$ of the vertex $v$, as indicated in Figure~\ref{fig:poset4}.

\begin{figure}[ht]
	\begin{center}
		\includegraphics[width=0.7\textwidth]{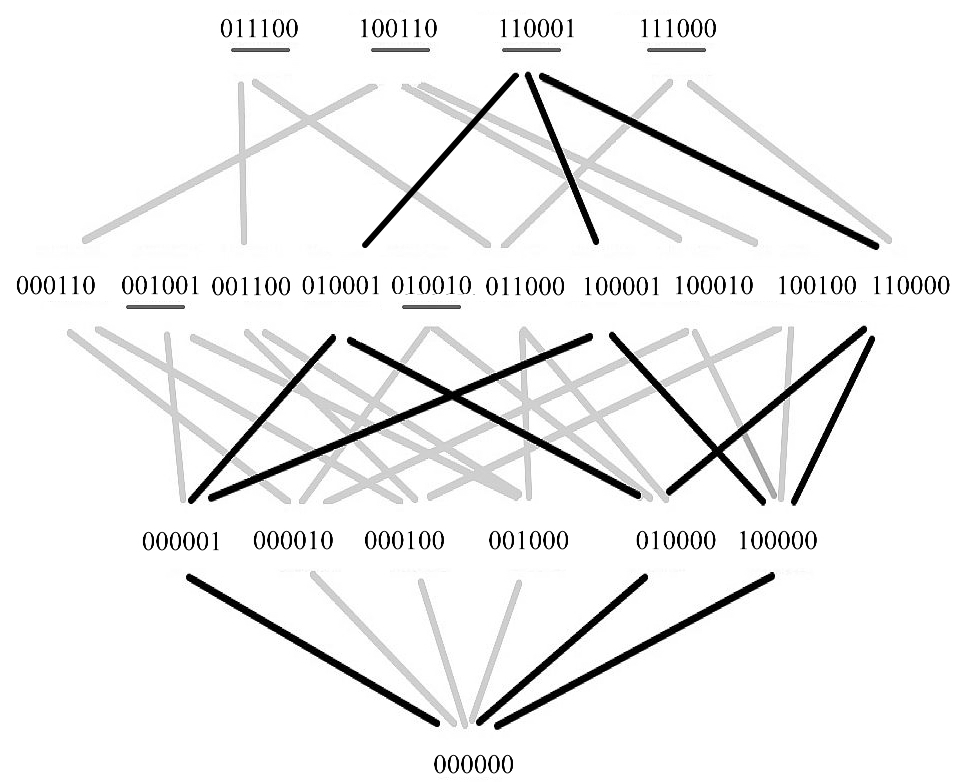}
	\end{center}
	\caption{The Hasse diagram of the graph $\R_6$ as a poset. 
		The trailing $00$ of the vertex labels have been 
		truncated. The dark lines show the interval $[000000, 110001]$ which is isomorphic to 
		the 3-dimensional hypercube $Q_3$. The six maximal elements of the poset are underlined.}
	\label{fig:poset6}
\end{figure}

Notice that $(\R_n, \leq )$ is a ranked poset, where the rank of $u \in \R_n$ 
is its Hamming weight $| u |_1$.
Since the number of vertices in $\R_n$ of weight $w$ is given by
\begin{equation}\label{weight_formula}
{n-w+1 \choose w}
\end{equation}
for $ 0 \leq w \leq \lceil n/2 \rceil$~\cite[Corollary 3.2]{paper1}, the rank 
generating polynomial of the poset $\R_n$ is 
\begin{equation}\label{rank_generating}
F(\R_n, x) = \sum_{k =0}^{\lceil n/2 \rceil} 
{n-k+1 \choose k} x^k ~.
\end{equation}

Next we consider the maximal elements of $\R_n$ and determine its M\"{o}bius function.

\subsection{Maximal elements}
The maximal elements of a poset are those elements which are not 
smaller than any other element of the set. The following table lists the run-constrained 
binary strings of lengths $3, 4, \ldots , 8$ which correspond 
to maximal elements of $ \R_1 $ through $\R_6$,  viewed as a poset. 

\begin{center}
	\begin{tabular}{llllll}
		$ \R_1 $ &
		$\R_2 $ &
		$\R_3 $ &
		$\R_4 $ &
		$\R_5 $ &
		$\R_6 $  \\ \hline
		100 & 0100 & 00100 & 011000 & 0011000 & 00100100 \\
		&1000 & 11000& 100100 &0100100 & 01001000 \\
		&&& 110000 & 1000100 & 01110000\\
		&&&&1001000 & 10011000 \\
		&&&&1110000 & 11000100 \\
		&&&&&11100000
	\end{tabular}
\end{center}

The sequence $\{ M_n \}_{n \geq 1}$ of the number of maximal elements of $ \R_n$ starts as 
$$
1,2,2,3,5,6,10,13,20,27,40, 56, 80, \ldots
$$
We can obtain the generating function of this sequence by setting $ d=1 $ 
and $u=0$ in the up-down 
degree enumerator generating function in Theorem~\ref{thm:up-down}. This specialization 
gives the following result.

\begin{corollary}\label{maximals}
If $M_n$ denotes the number of maximal elements of the 
poset of $ \R_n$, then
the generating function of 
the sequence $\{ M_n \}_{n \geq 1}$ is given by
$$
\sum_{n \geq 1 } M_n t^ n = 
\frac{
	t ( 1+2t-2t^3+t^5+t^6)
}
{
	1 - 2 t^2 - t^3 + t^4 + t^5 - t^7
} ~.
$$
\end{corollary}

\subsection{The M\"{o}bius function}
We  can easily compute the M\"{o}bius function $\mu$ of the 
poset $(\R_n, \leq )$, since every interval 
$[u, v]$ is isomorphic to a cube (Boolean algebra) and therefore 
has the same M\"{o}bius function. Denoting the weights of $u$ and $v$ by 
$|u|_1$ and $|v|_1$, we have
$$
\mu (u, v ) =
\left\{
\begin{array}{ll}
(-1)^{|v|_1 - |u|_1 } & \mbox{ if } ~u \leq v ,\\
0 & \mbox{ if }  ~u \not\leq v .
\end{array}
\right.
$$

\section{Inversion generating function}
\label{sec:inv}
In this section, we present an observation about a combinatorial property of run-constrained binary strings that 
we use to define Fibonacci-run graphs.
Let us denote by $RC_n $ run-constrained binary strings of length $n \geq 0$. 
So for instance, adding the trailing 00 to the labels of the vertices of $\R_1 - \R_3$ shown in 
Figure~\ref{fig:fibonacci_run2}, gives 
\begin{eqnarray*}
RC_3 &  = &  \{000,100\} ~, \\
RC_4 &  = &  \{0000,0100,1000 \} ~, \\
RC_5 &  = &  \{00000, 00100, 01000, 10000, 11000 \} ~.
\end{eqnarray*}
For $ n \geq 0$, consider the inversion enumerator polynomial $ Q_n(x,q)$ defined by
\begin{equation}\label{inv}
Q_n(x, q) = \sum_{w \in RC_n} x^{|w|_1} q^{\inv(w)},
\end{equation}
where $\inv(w)$ denotes the number of inversions of $w$ and 
$|w|_1 $ is the Hamming weight, or the rank of $w$.
Recall that for $w = w_1 w_2 \cdots w_n$, $\inv(w)$ is the number 
of pairs $1 \leq i < j \leq n$ with $ w_i > w_j$.
If we omit to write down the dependence on $q$ in $Q(x,q)$ for notational convenience,
then $Q_0(x) = Q_1(x) = Q_2(x)=1$, and

\begin{eqnarray*}
	Q_3(x) &=& 1+ x,\\
	Q_4(x) &=& 1+ (q+1) x,\\
	Q_5(x) &=& 1+ \left(q^2+q+1\right) x+x^2,\\
	Q_6(x) &=& 1 +\left(q^3+q^2+q+1\right) x + \left(2 q^2+1\right) x^2,\\
	Q_7(x) &=& 1 +\left(q^4+q^3+q^2+q+1\right) x  + \left(2 q^4+q^3+2 q^2+1\right) x^2+x^3 ~.\\
\end{eqnarray*}

\begin{proposition}
	With $Q_0(x) = Q_1(x) = Q_2(x)=1$ and $ Q_{-n}(x) = 0$, we have 
	\begin{equation}\label{invg}
	Q_n(x) = \sum_{ k \geq 0} x^k Q_{n-1-2k } (x q^{k+1}).
	\end{equation}
\end{proposition}
\begin{proof}
	The proof is a consequence of the fundamental decomposition of Fibonacci-run graphs given in
Theorem~\cite[Lemma 4.1]{paper1}. Note that the terms in the 
	sum vanish for $ k > \lfloor \frac{n-1}{2} \rfloor$.
\end{proof}

Define
$$
H(x, t ) = \sum_{n \geq 0} Q_n (x) t^n ~.
$$
Then a consequence of~\eqref{invg} (writing them out with $x$ replaced by $xq^k$ for $k \geq 0$ 
and summing up by columns) is the functional identity
\begin{equation}\label{functional}
H(x, t ) = 1+ t \sum_{k \geq 0} x^k  t^{2k} H(xq^{k+1} ,t) ~.
\end{equation}
Moreover, substituting $q=1$ in $Q_n(x)$ gives the rank generating polynomial
$F(\R_n, x )$ of $\R_n$ defined in~\eqref{rank_generating}.
In other words, the coefficient of $x^k $ becomes the number of words in $RC_n$ with 
weight $k$ as given by~\eqref{weight_formula}. For instance, for $n=9$ and $ q=1$, we have
$$
1 + 7 x + 15 x^2 + 10 x^3 +x^4,
$$
for which the coefficients are
$$
{8 \choose 0 } = 1,~
{7 \choose 1 } = 7,~
{6 \choose 2 } = 15,~
{5 \choose 3 } = 10,~
{4 \choose 4 } = 1.
$$

Taking $q=1$ in~\eqref{functional}, we see that $H$ satisfies
$$
H(x, t ) = 1+  \frac{t H(x,t)}{1- x t^2}~,
$$
so that
$$
H(x,t) = \frac{1-x t^2}{1-t-x t^2} ~.
$$
After subtracting off the terms $ 1, t, t^2$ corresponding to the null word, $0$ and $00$, and 
dividing by $t^2$ to remove the contribution of the trailing $00$ in the strings of $RC_n$ to the 
length, we get the generating function of 
the rank generating polynomials $ F(\R_n, x)$ (see~\eqref{rank_generating}) 
of Fibonacci-run graphs as a poset as
\begin{eqnarray*}
\frac{t ( 1+x + x t )}{1-t- x t^2} & = & \sum_{n \geq 1} F(\R_n, x) t^n \\
&=& (1+x) t + (1 + 2 x) t^2 + (1+3x+x^2)t^3+ (1+4x + 3 x^2) t^4 +\cdots
\end{eqnarray*}

Inversions and the major index  statistics of a similar flavor  for a class 
of related Fibonacci strings can be found in~\cite{Egecioglu2020TOC}.

\section{Embedding related results}
\label{sec:encoding}
We can encode a binary string of length $n$ as a run-constrained binary string of length $3n+1$.
Let $s_i = 1^i 0^{i+1}$ for $ i \geq 1$. Given a binary string with $k$ runs of $1$s, 
$$
w= 
0^{j_0} 1^{i_1} 0^{j_1} 1^{i_2} \cdots 1^{i_k} 0^{j_k},
$$
we first encode the runs as 
$$
0 s_{j_0} s_{i_1} s_{j_1} \cdots  s_{i_k}s_{j_k}
$$
if 
$ j_0 >0$ (i.e. the binary string starts with 0),  and by 
$$
s_{i_1} s_{j_1} \cdots  s_{i_k}s_{j_k}
$$
if $j_0 = 0$ (i.e. the binary string starts with 1). Then, if necessary, we append $0$s at the end to make the length of the encoding $3n+1$.

For $n=3$, this works as follows
$$
\begin{array}{lclcl}
000 & \rightarrow & 0 s_3 & \rightarrow & 01110000 \, 00 \\
001 & \rightarrow & 0 s_2 s_1 &\rightarrow &  011000100 \, 0\\
010 & \rightarrow & 0s_1 s_1 s_1 & \rightarrow & 0100100100\\
100 & \rightarrow & ~ s_1 s_2 & \rightarrow & 10011000 \, 00 \\
011 &\rightarrow &  0 s_1 s_2 & \rightarrow & 010011000 \, 0\\
101 &\rightarrow & ~ s_1s_1s_1 & \rightarrow & 100100100 \, 0 \\
110 &\rightarrow &  ~ s_2 s_1 & \rightarrow & 11000100 \, 00\\
111 & \rightarrow & ~ s_3 & \rightarrow & 1110000 \, 000
\end{array}
$$

The described encoding embeds the hypercube $Q_n$ into $\R_{3n-1}$ (the trailing $00$ is not included in the vertex labels here). 

\begin{question}\label{question:dilation}
	What is the dilation of this embedding, i.e.  the value of 
$$
\max_{uv \in E(Q_n)} d_{\R_{3n-1}} (u',v') ~,
$$ 
where the prime denotes the image under the encoding above?
\end{question}

We note that although it is intuitive, the embedding described above does not seem to give the 
the smallest dimensional Fibonacci--run graph in which it is possible to embed $Q_n$.\\

The study
of hypercubes of various dimensions which are subgraphs of a Fibonacci-run graphs are interesting 
in its own right. An analogous 
question has been studied for Fibonacci cubes, cf.\ results about cube polynomial 
listed in~\cite{klavzar2013-survey}, and generalizations in~\cite{Saygi2017} .

Let $h_{n,k}$ denote the number of $k$-dimensional hypercubes $Q_k$ in $\R_n$.
A corollary of~\cite[Proposition 8.2]{paper1}, obtained by taking $q=1$ in 
that proposition is the following. The generating function of the cube 
polynomials of $\R_n$ is given by

{\small
\begin{equation}\label{cubeGF}
\sum_{n\geq 1} t^n \sum_{  k \geq 0 } h_{n,k} x^k  = 
\frac{t (2+x +(x+1)t +  x (x+2)t^2  +x(x+1)t^3 + x (x+1)t^4}
{1 -t -t^2 -xt^3 - x (x+1) t^5 } ~.
\end{equation}
}

In the series expansion of this generating function in powers of $t$, the 
largest $m$ for which the term $x^m$ appears as a coefficient of $t^n$, gives the 
dimension of the largest hypercube $Q_m$ that can be embedded in $\R_n$.
Calculations on~\eqref{cubeGF}, using high order derivatives with respect to $x$ 
(with Mathematica) suggest  that 
for $ m \geq 0$,  the hypercube $Q_ {2 m + 1}$ embeds in $\R_{5 m + 1}$ and
the hypercube $Q_ {2 m + 2}$ embeds in $\R_{5 m + 3}$, 
and these are the smallest possible Fibonacci-run graphs with this  property.  So it appears that 
the hypercube graph $ Q_n$ can be embedded into 
$\R_{\lceil (5n-4)/2 \rceil}$, 
and this is the smallest possible run graph with this property. 

\begin{conjecture}\label{conjecture:embedding}
The smallest $m$ for which 
the hypercube graph $ Q_n$ can be embedded into  $\R_m$ is 
$m= \big\lceil \frac{5n-4}{2} \big\rceil$. 
\end{conjecture}

\ignore
{
Recall that the isometric dimension of a partial cube is the minimum dimension of a hypercube onto which it may be isometrically embedded (i.e. distances in the original graph are also distances in the target graph), and is equal to the number of equivalence classes of the Djokovic-Winkler relation. *** add references, maybe even write out the relation? ***

The isometric dimension of $\R_n$ and $ \R_n^*$ (i.e.  $\R_n$ 
without the all zero vertex) 
was calculated on Mathematica by
constructing the equivalence classes of the Dvokovic-Winkler theorem 
(Run\_graph\_isometric\_dimension\_May9\_2020.nb)
for $n \leq 13$. We have the following table

\begin{center}
	\begin{tabular}{|l|c|c|c|c|c|c|c|c|c|c|c|c|} \hline \hline
		$n$ & 2& 3 & 4 & 5 & 6 & 7 & 8 & 9 & 10 & 11 & 12 &13 \\ \hline
		dimension of $\R_n$ & 2& 3 & 4 & 5 & 6 & 7 & 8 & 9 & 10 & 11 & 12 & 13\\ \hline
		dimension of $\R_n^*$ &1 & 2 & 6 & 6 & 6 & 8 & 9 & 10 & 11 & 12 & 12& 13 \\ \hline
	\end{tabular}
\end{center}

It is not clear what the second row suggests for the isometric dimension 
of $\R_n^*$.
Is it $n$ or $n+1$ depending on $n$, or does it stabilize to $n$ eventually?
}

\section{Further directions}
\label{sec:further}

In the final section, we list various questions and conjectures, which are of interest in the further study of 
Fibonacci-run graphs. These are in addition to 
the Conjecture~\ref{conjecture:embedding} on the determination of the smallest dimensional Fibonacci-run graph 
that contains $Q_n$, 
the Question~\ref{question:dilation} 
on the dilation of the mapping described at the start of Section~\ref{sec:encoding}, 
and the Conjecture~\ref{conjecture:GF} on the form of the generating function of 
the number of vertices of a given degree in $\R_n$. 

For Conjecture~\ref{conjecture:GF}, 
the extraction of the coefficients in  the generating function 
of the degree enumerator polynomials $f(t,x)$ of Theorem~\ref{degreegf}, one may  
use the mechanism described in Remark~\ref{degree_remark}, and the 
examples presented afterwards, along with the Leibniz formula for higher derivatives and the 
reciprocal differentiation result in~\cite{leslie1991}.

Considering the results obtained in Examples~\ref{ex:deg2} -- \ref{ex:deg5} in Section~\ref{sec:con}, the following general question arises.

\begin{question}
What is the number of vertices of degree $k$ in $\R_n$?
\end{question}

The 
analysis and the conjecture  on the diameter of $\R_n$ can be found in~\cite{paper1}. In relation 
to this, a natural question that arises is the following:
\begin{question}\label{question:radius}
What is the 
radius of $\R_n$? 
\end{question}
It may be possible to consider Question~\ref{question:radius} in conjunction with the analysis for the 
the exact diameter
of $\R_n$ in~\cite[Section 5]{paper1}.\\ 

An irregularity measure of graphs was defined by Albertson~\cite{Albertson},
and recently 
studied for various families of graphs in~\cite{Klavzar_irr}, and generalized to a polynomial 
enumerator in~\cite{irr_ESS}. 

\begin{question}
What is the 
the irregularity, or more generally the 
irregularity polynomial, of $ \R_n$, as defined in~\cite{irr_ESS}?
\end{question}

Finally, there is the problem of the determination of the smallest dimensional 
hypercube into which $ \R_n$ can be isometrically embedded. In other words

\begin{question}
What is the isometric dimension of $\R_n$? 
\end{question}

It seems likely that 
the isometric dimension of $ \R_n$ is as low as a linear function of $n$. \\


\section*{Acknowledgments}
We would like to thank Prof. Sandi Klav\v{z}ar 
for enabling the cooperation of the authors of this article, and his early interest in the topic.
The first author would like to acknowledge the hospitality of Reykjavik University during his 
sabbatical stay there, during which a portion of this research was carried out.

\bibliographystyle{acm}
\bibliography{mybib}{}

\end{document}